\documentclass{amsart}

\usepackage[english]{babel}
\usepackage{hyperref}

\usepackage[utf8]{inputenc}
\usepackage[T1]{fontenc}
\usepackage{pifont}

\usepackage[margin=3cm]{geometry}

\usepackage{amsmath}
\usepackage{amssymb}
\usepackage{amsthm}
\usepackage{amscd}
\usepackage{mathrsfs}
\usepackage[abbrev,backrefs]{amsrefs}

\usepackage{tikz-cd}

\usepackage{graphicx}

\usepackage {cancel}

\numberwithin{equation}{section}

\newtheorem{theorem}{Theorem}[section]	
\newtheorem{proposition}[theorem]{Proposition}	
\newtheorem{corollary}[theorem]{Corollary}	
\newtheorem{lemma}[theorem]{Lemma}
\newtheorem{definition}[theorem]{Definition}

\newtheorem{remark}[theorem]{Remark}

\newcommand{\A}{\mathscr{A}}
\newcommand{\B}{\mathscr{B}}
\newcommand{\C}{\mathscr{C}}

\newcommand{\F}{\mathscr{F}}
\newcommand{\z}{\mathscr{Z}}

\newcommand{\E}{\mathscr{E}}

\begin{document}

\title{\textsc{Central extensions of preordered groups}}
\author{Marino Gran}
\author{Aline Michel}
\thanks{The second author's research is funded by a FRIA doctoral grant of the \emph{Communaut\'e française de Belgique}}
\email{marino.gran@uclouvain.be}
\email{aline.michel@uclouvain.be}
\address{Universit{\'e} Catholique de Louvain, Institut de Recherche en Math{\'e}matique et Physique, Chemin du Cyclotron 2, 1348 Louvain-la-Neuve, Belgium}
\date{}
\begin{abstract}
We prove that the category of preordered groups contains two full reflective subcategories that give rise to some interesting Galois theories. The first one is the category of the so-called commutative objects, which are precisely the preordered groups whose group law is commutative. The second one is the category of abelian objects, that  turns out to be the category of monomorphisms in the category of abelian groups. We give a precise description of the reflector to this subcategory, and we prove that it induces an admissible Galois structure and then a natural notion of \emph{categorical central extension}. We then characterize the central extensions of preordered groups in purely algebraic terms: these are shown to be the central extensions of groups having the additional property that their restriction to the positive cones is a special homogeneous surjection of monoids.
\end{abstract}
\keywords{Preordered groups, preordered abelian groups, categorical Galois theory, abelian object, commutative object, central extension, special homogeneous surjection.}
\maketitle

%Introduction
\section*{Introduction}
A \emph{preordered group} $(G,\le)$ is a group $G = (G,+,0)$ endowed with a preorder relation $\le$ which is compatible with the addition $+$ of the group $G$, in the sense that, if $a \le c$ and $b \le d$, then $a + b \le c + d$  (for $a,b,c,d \in G$). Preordered groups and monotone group homomorphisms form a category, denoted by $\mathsf{PreOrdGrp}$.
This category is actually isomorphic to another category, whose objects are given by pairs $(G,P_G)$, where $G$ is a group and $P_G$ a submonoid of $G$ closed under conjugation in $G$. This submonoid $P_G$ is usually called the \emph{positive cone} of $G$, and an object $(G,P_G)$ can be depicted as \begin{center}
\begin{tikzcd}
P_G \arrow[r,tail]
& G
\end{tikzcd}
\end{center} 
where the arrow represents the inclusion of $P_G$ in $G$. An arrow between two such objects $(G,P_G)$ and $(H,P_H)$ is given by a pair $(f,\bar{f}) \colon (G,P_G) \rightarrow (H,P_H)$ of monoid morphisms making the diagram 
\begin{equation} \label{morphism in PreOrdGrp}
\begin{tikzcd}
P_G \arrow[r,"{\bar{f}}"] \arrow[d,tail]
& P_H \arrow[d,tail]\\
G \arrow[r,"f"']
& H
\end{tikzcd}
\end{equation}
commute, so that $f \colon G \rightarrow H$ is a group homomorphism that ``restricts'' to the positive cones, in the sense that $f(P_G) \subseteq P_H$.
It is this second equivalent definition of the category $\mathsf{PreOrdGrp}$ of preordered groups that we shall use throughout this article.
As shown in \cite{CMFM} the category $\mathsf{PreOrdGrp}$ is  both complete and cocomplete and is a \emph{normal} category in the sense of  \cite{JZ10}, that is a pointed regular category where every regular epimorphism is a normal epimorphism (i.e. a cokernel). %Note that, in any normal category, a regular epimorphism is always the cokernel of its kernel.

In this article we prove that the lattice of normal subobjects on any preordered group $(G,P_G)$ is \emph{modular} (Proposition \ref{Modularity of PreOrdGrp}), and this implies that any reflective subcategory of $\mathsf{PreOrdGrp}$ that is also closed in it under subobjects and regular quotients is \emph{admissible} from the point of view of Categorical Galois Theory \cite{JG90} (see Proposition \ref{prop JK}). In particular the full subcategory $\mathsf{PreOrdAb}$ of preordered \emph{abelian} groups satisfies this property, giving rise to the adjunction 
\begin{equation} 
\begin{tikzcd}
\mathsf{PreOrdGrp} \arrow[rr,bend left=9,"C"]
& \bot
& \mathsf{PreOrdAb} \arrow[ll,hook',bend left=9,"V"]
\end{tikzcd}
\end{equation}
where $V$ is the inclusion functor and its left adjoint $C$ sends a preordered group $(G,P_G)$ to the preordered abelian group $(G/[G,G], \eta_G (P_G))$, where $\eta_G \colon G \twoheadrightarrow G/[G,G]$ is the quotient of $G$ by its derived subgroup $[G,G]$. Preordered abelian groups turn out to be precisely the \emph{commutative objects} (in the sense of \cite{BB}) of the category $\mathsf{PreOrdGrp}$. A characterization of the normal extensions of preordered groups with respect to this adjunction is given in Theorem \ref{GammaC-normal extensions}: these are precisely the normal epimorphisms 
\begin{tikzcd}
(f,\bar{f}) \colon (G,P_G) \arrow[r,two heads]
& (H,P_H)
\end{tikzcd} 
such that 
\begin{tikzcd}
f \colon G \arrow[r,two heads]
& H
\end{tikzcd} 
is a central extension of groups and, moreover, $a - b + c \in P_G$ whenever $a,b,c \in P_G$ are such that  $\eta_G(a) = \eta_G(b)$ and $f(b) = f(c)$.

We then turn our attention to the composite adjunction 

\begin{center}
\begin{tikzcd}
\mathsf{PreOrdGrp} \arrow[rr,bend left=7.5,"C"]
& \bot
& \mathsf{PreOrdAb} \arrow[ll,hook',bend left=7.5,"V"] \arrow[rr,bend left=7.5,"A"]
& \bot
& \mathsf{Mono(Ab)} \arrow[ll,hook',bend left=7.5,"W"]
\end{tikzcd}
\end{center}
 where $\mathsf{Mono(Ab)}$ is the category of monomorphisms in the category $\mathsf{Ab}$ of abelian groups, $W$ is the inclusion functor and $A$ its left adjoint (described in detail in Section \ref{Commutative and abelian objects}).
 We prove that $\mathsf{Mono(Ab)}$ is the category of \emph{abelian objects} in $\mathsf{PreOrdGrp}$ (Corollary \ref{AbelianObjects}), and we characterize the normal epimorphisms 
\begin{equation} \label{extensions in PreOrdGrp}
\begin{tikzcd}
P_G \arrow[r, two heads, "{\bar{f}}"] \arrow[d,tail]
& P_H \arrow[d,tail]\\
G \arrow[r, two heads,"f"']
& H
\end{tikzcd}
\end{equation}
in $\mathsf{PreOrdGrp}$ (i.e. both $f$ and $\bar{f}$ are surjective) that are \emph{central extensions} in the sense of Categorical Galois Theory \cite{JK}  for this composite adjunction. By using the results established in \cite{MRVdL} we show in Theorem \ref{Gamma-normal and Gamma-central extensions} that these are characterized by the fact that the surjective morphism $f$ is a central extension of groups and $\bar{f}$ a special homogeneous surjection in the sense of \cite{BMFMS}. This result opens the way to the possibility of studying the non-abelian homology of preordered groups by using the approach adopted in \cite{DEG} (which is itself based on the one in \cite{GJ-Hopf}), that we leave for future work.

%We conclude this section by mentioning the fact that the category $\mathsf{PreOrdGrp}$ of preordered groups is \emph{unital} \cite{BB} (see Section \ref{Commutative and abelian objects} for the definition). This result is also due to Clementino, Martins-Ferreira and Montoli \cite{CMFM}.  

%Categorical Galois theory
\section{Categorical Galois structures and central extensions}

We first recall some definitions and results of categorical Galois theory which will be needed for our work. For this section, we mainly follow \cite{JG90,JG,JK}. 

\begin{definition} \label{definition Galois structure}
A \emph{Galois structure} is a system $\Gamma = (\C,\F,F,U,\E,\z)$ in which
\begin{itemize}
\item 
\begin{tikzcd}
\C \arrow[rr,bend left=15,"F"]
& \bot
& \F \arrow[ll,bend left=15,"U"]
\end{tikzcd}
is an adjunction, with unit $\eta$ and counit $\epsilon$;
\item $\E$ and $\z$ are classes of morphisms in $\C$ and $\F$, respectively,
\end{itemize}
such that
\begin{itemize}
\item $\C$ and $\F$ admit all pullbacks along morphisms from $\E$ and $\z$, respectively;
\item $\E$ and $\z$ are closed under composition, contain all isomorphisms and are pullback-stable;
\item $F(\E) \subseteq \z$;
\item $U(\z) \subseteq \E$.
\end{itemize}
\end{definition} 

Let $\Gamma = (\C,\F,F,U,\E,\z)$ be a Galois structure. For any object $B$ in $\C$, let $\E(B)$ denote the full subcategory of the slice category $\C \downarrow B$ determined by the morphisms
\begin{tikzcd}
A \arrow[r,"f"]
& B
\end{tikzcd}
in the class $\E$. Objects in this subcategory are called \emph{extensions} of $B$ and are denoted by $(A,f)$. Let $p \colon E \rightarrow B$ be any arrow in $\C$. Then $p^* \colon \E(B) \rightarrow \E(E)$ is the change-of-base functor associating, with any object $f \colon A \rightarrow B$ in $\E(B)$, the object $p^*(f) = \pi_1 \colon E \times_B A \rightarrow E$ as in the following pullback diagram:
\begin{center}
\begin{tikzcd}
E \times_B A \arrow[r,"{\pi_2}"] \arrow[d,"{\pi_1}"']
& A \arrow[d,"f"]\\
E \arrow[r,"p"']
& B.
\end{tikzcd}
\end{center}

It is well-known that a Galois structure $\Gamma = (\C,\F,F,U,\E,\z)$ induces, for any object $B \in \C$, an adjunction between the categories of extensions as in the diagram
\begin{equation} \label{adjunction induced by Galois structure}
\begin{tikzcd}
\E(B) \arrow[rr,bend left=12,"F^B"]
& \bot
& \z(F(B)) \arrow[ll,bend left=12,"U^B"]
\end{tikzcd}
\end{equation}
with unit and counit denoted by $\eta^B$ and $\epsilon^B$, respectively. The left adjoint $F^B \colon \E(B) \rightarrow \z(F(B))$ is defined, for any $(A,f) \in \E(B)$, by $F^B(A,f) = (F(A),F(f)) (\in \z(F(B)))$, while the right adjoint $U^B \colon \z(F(B)) \rightarrow \E(B)$ sends any $(X,\phi) \in \z(F(B))$ to the pullback $\eta_B^*(U(\phi))$ of $U(\phi)$ along $\eta_B$:
\begin{center}
\begin{tikzcd}
B \times_{UF(B)} U(X) \arrow[r] \arrow[d,"{\eta_B^*(U(\phi))}"']
& U(X) \arrow[d,"{U(\phi)}"]\\
B \arrow[r,"{\eta_B}"']
& UF(B).
\end{tikzcd}
\end{center}

\begin{definition}
A Galois structure $\Gamma = (\C,\F,F,U,\E,\z)$ is \emph{admissible} when, for any $B \in \C$, the counit morphism $\epsilon^B$ of the adjunction \eqref{adjunction induced by Galois structure} is an isomorphism.
\end{definition}

There is an equivalent way to define the admissibility of a Galois structure $\Gamma = (\C,\F,F,U,\E,\z)$, which corresponds to the equivalence $(1) \Leftrightarrow (2)$ in the following Proposition. Under an additional condition on the Galois structure, we also obtain the equivalence with $(3)$. This extra condition is that the counit $\epsilon$ of the adjunction $F \dashv U$ is an isomorphism. In this paper, we shall always be in such a situation, so that it will be possible to use this last equivalent definition of the admissibility.

\begin{proposition}\label{BasicAdj}
Let $\Gamma = (\C,\F,F,U,\E,\z)$ be a Galois structure such that the counit $\epsilon$ of the adjunction $F \dashv U$ is an isomorphism. Then, the following conditions are equivalent:
\begin{enumerate}
\item $\Gamma$ is admissible;
\item for any $B \in \C$, the functor $U^B \colon \z(F(B)) \rightarrow \E(B)$ is fully faithful;
\item $F$ preserves all pullbacks of the form
\begin{center}
\begin{tikzcd}
B \times_{UF(B)} U(X) \arrow[r,"{\pi_2}"] \arrow[d,"{\pi_1}"']
& U(X) \arrow[d,"{U(\phi)}"] \\
B \arrow[r,"{\eta_B}"']
& UF(B)
\end{tikzcd}
\end{center}
where $\phi \in \z$.
\end{enumerate}
\end{proposition}

From now on, we assume that an admissible Galois structure $\Gamma = (\C,\F,F,U,\E,\z)$ as in Proposition \ref{BasicAdj} has been fixed. Let us then recall the notions of ($\Gamma$-)\emph{trivial}, ($\Gamma$-)\emph{central} and ($\Gamma$-)\emph{normal extensions}.

\begin{definition}
A morphism $f \colon A \rightarrow B$ in $\E$ is a \emph{($\Gamma$-)trivial extension} when the square
\begin{center}
\begin{tikzcd}
A \arrow[r,"{\eta_A}"] \arrow[d,"f"']
& UF(A) \arrow[d,"{UF(f)}"]\\
B \arrow[r,"{\eta_B}"']
& UF(B)
\end{tikzcd}
\end{center}
is a pullback. Equivalently, $f \colon A \rightarrow B$ in $\E$ is a \emph{($\Gamma$-)trivial extension} when $f$ lies in the image of the functor $U^B \colon \z(F(B)) \rightarrow \E(B)$.
\end{definition}

Remark that the equivalence in the above definition follows directly from the admissibility of the Galois structure $\Gamma$.

By a \emph{monadic extension} we mean
a morphism $p \colon E \rightarrow B$ in $\E$ that is also an \emph{effective descent morphism} (see \cite{JST,JT}, for instance).

%\begin{remark}
%\emph{We do not recall the (general) definition of an effective descent morphism since, in our context of preordered groups, an effective descent morphism is simply given by a regular epimorphism whose description is well-known (see Section \ref{The category PreOrdGrp}).}
%\end{remark}

\begin{definition}
A morphism $f \colon A \rightarrow B$ in $\E$ is a \emph{($\Gamma$-)central extension} when there exists a monadic extension $p \colon E \rightarrow B$ such that $p^*(f) \colon E \times_B A \rightarrow E$ is a $(\Gamma\text{-})$trivial extension, that is, the following diagram
\begin{center}
\begin{tikzcd}[row sep = large,column sep=large]
E \times_B A \arrow[r,"{\eta_{E \times_B A}}"] \arrow[d,"{p^*(f) = \pi_1}"']
& UF(E \times_B A) \arrow[d,"{UF(\pi_1)}"]\\
E \arrow[r,"{\eta_E}"']
& UF(E)
\end{tikzcd}
\end{center}
is a pullback, where $\pi_1$ is the first projection in the pullback
\begin{center}
\begin{tikzcd}
E \times_B A \arrow[r,"{\pi_2}"] \arrow[d,"{p^*(f) = \pi_1}"']
& A \arrow[d,"f"]\\
E \arrow[r,"p"']
& B.
\end{tikzcd}
\end{center}
\end{definition}

\begin{definition}
An arrow $f \colon A \rightarrow B$ in $\E$ is a \emph{($\Gamma$-)normal extension} when $f$ is a monadic extension and $f^*(f)$ is a $(\Gamma\text{-})$trivial extension.
\end{definition}

Clearly, any trivial extension is central and any normal extension is central. The admissibility of the Galois structure $\Gamma$ also guarantees that any trivial extension is normal.

%Functor C
\section{The reflector $C \colon \mathsf{PreOrdGrp} \rightarrow \mathsf{PreOrdAb}$ and its induced admissible Galois structure} \label{Functor C}

The thorough study of the properties of the category $\mathsf{PreOrdGrp}$ carried out in \cite{CMFM} provides in particular a description of the limits and the colimits of this category that will be needed for our work.  
We briefly recall these constructions for the reader's convenience.

First, the product of two preordered groups $(G,P_G)$ and $(H,P_H)$ is given by the preordered group $(G \times H,P_G \times P_H)$. Then the equalizer of two parallel arrows
\begin{equation} \label{parallel arrows}
\begin{tikzcd}
(G,P_G) \arrow[r,shift left,"{(f,\bar{f})}"] \arrow[r,shift right,"{(g,\bar{g})}"']
& (H,P_H)
\end{tikzcd}
\end{equation}
is given by $((E,P_E),(e,\bar{e}))$ where $(E,e)$ is the equalizer of $f$ and $g$ in the category $\mathsf{Grp}$ of groups and $P_E = E \cap P_G$, that is, the following square is a pullback in $\mathsf{Mon}$:
\begin{center}
\begin{tikzcd}
P_E \arrow[r,tail,"{\bar{e}}"] \arrow[d,tail]
& P_G \arrow[d,tail]\\
E \arrow[r,tail,"e"']
& G.
\end{tikzcd}
\end{center}
It is not difficult to see that this is equivalent to saying that $(E,e)$ is the equalizer of $f$ and $g$ in $\mathsf{Grp}$ and $(P_E,\bar{e})$ the equalizer of $\bar{f}$ and $\bar{g}$ in $\mathsf{Mon}$. As a consequence, the pullback of two morphisms $(f,\bar{f}) \colon (G,P_G) \rightarrow (H,P_H)$ and $(g,\bar{g}) \colon (C,P_C) \rightarrow (H,P_H)$ with the same codomain $(H,P_H)$ is given by $((P,P_P),(p_1,\bar{p}_1),(p_2,\bar{p}_2))$ where $(P,p_1,p_2)$ is the pullback of $f$ and $g$ in $\mathsf{Grp}$ and $(P_P,\bar{p}_1,\bar{p}_2)$ the pullback of $\bar{f}$ and $\bar{g}$ in $\mathsf{Mon}$. More generally, all limits of $\mathsf{PreOrdGrp}$ are computed “componentwise”. Note that in particular the kernel $(K,P_K)$ of a morphism $(f,\bar{f}) \colon (G,P_G) \rightarrow (H,P_H)$ in $\mathsf{PreOrdGrp}$ can also be computed by taking $K$ to be the kernel $\mathsf{Ker}(f)$ of $f$ in $\mathsf{Grp}$ and $P_K$ the intersection $K \cap P_G$.

Colimits in $\mathsf{PreOrdGrp}$ are a bit more difficult to describe, in general, since they are not simply computed “componentwise”. We shall be mainly interested in coequalizers: given two parallel arrows as in \eqref{parallel arrows} their coequalizer is given by $((Q,P_Q),(q,\bar{q}))$ with $(Q,q)$ the coequalizer of $f$ and $g$ in $\mathsf{Grp}$ and $P_Q = q(P_H)$ (i.e. $\bar{q}$ is surjective):
\begin{center}
\begin{tikzcd}
P_H \arrow[r,two heads,"{\bar{q}}"] \arrow[d,tail]
& P_Q \arrow[d,tail]\\
H \arrow[r,two heads,"q"']
& Q.
\end{tikzcd}
\end{center}
In particular the cokernel of a morphism $(f,\bar{f}) \colon (G,P_G) \rightarrow (H,P_H)$ in $\mathsf{PreOrdGrp}$ is then given by a pair $(q,\bar{q})$ making the diagram above commute, with $q$ the cokernel of $f$ in $\mathsf{Grp}$ and $\bar{q}$ surjective.

As recalled in the Introduction, an important result of \cite{CMFM} is that the category $\mathsf{PreOrdGrp}$ of preordered groups is \emph{normal}. 
Although the category $\mathsf{PreOrdGrp}$ is not \emph{Barr-exact} \cite{Barr} (see \cite[Remark 2.6]{CMFM}) its \emph{effective descent morphisms} (see \cite{JST}, for instance) are easy to characterize. In this context effective descent morphisms coincide with normal epimorphisms, so that an effective descent morphism in the category $\mathsf{PreOrdGrp}$ of preordered groups has a fairly simple description: it is a morphism $(f,\bar{f}) \colon (G,P_G) \rightarrow (H,P_H)$ as in \eqref{morphism in PreOrdGrp} where both $f$ and $\bar{f}$ are surjective. Epimorphisms in $\mathsf{PreOrdGrp}$ are given by morphisms $(f,\bar{f})$ where only $f$ is required to be surjective, and monomorphisms are morphisms $(f,\bar{f})$ where $f$ is injective (which, in turn, implies that also $\bar{f}$ is injective).

Let us then denote by $\mathsf{PreOrdAb}$ the full subcategory of $\mathsf{PreOrdGrp}$ whose objects $(G,P_G)$ are preordered abelian groups, i.e. such that $G$ is abelian.
\begin{proposition}
There is an adjunction
\begin{equation} \label{adjunction PreOrdGrp_PreOrdAb}
\begin{tikzcd}
\mathsf{PreOrdGrp} \arrow[rr,bend left=9,"C"]
& \bot
& \mathsf{PreOrdAb} \arrow[ll,hook',bend left=9,"V"]
\end{tikzcd}
\end{equation}
between the category $\mathsf{PreOrdGrp}$ of preordered groups and its full subcategory $\mathsf{PreOrdAb}$ of preordered abelian groups, where the right adjoint $V$ is the inclusion functor and the left adjoint $C$ is defined, for any $(G,P_G) \in \mathsf{PreOrdGrp}$, by $C(G,P_G) = (G/[G,G],\eta_G(P_G))$
\begin{center}
\begin{tikzcd}
P_G \arrow[r,two heads,"\bar{\eta}_G"] \arrow[d,tail]
& \eta_G(P_G) \arrow[d,tail] \\
G \arrow[r,two heads,"\eta_G"']
& G/[G,G] = ab(G),
\end{tikzcd}
\end{center}
where $\eta_G(P_G)$ is the direct image of $P_G$ along the quotient $\eta_G$ of $G$ by its derived subgroup $[G,G]$.
\end{proposition}

\begin{proof}
Let $(G,P_G)$ be any preordered group. Then the $(G,P_G)$-component of the unit of the adjunction \eqref{adjunction PreOrdGrp_PreOrdAb} is given by the above morphism $(\eta_G,\bar{\eta}_G)$. Indeed, consider any preordered abelian group $(A,P_A)$ and any morphism $(f,\bar{f}) \colon (G,P_G) \rightarrow (A,P_A)$ in $\mathsf{PreOrdGrp}$:
\begin{center}
\begin{tikzcd}
P_G \arrow[rr,two heads,"\bar{\eta}_G"] \arrow[dr,"\bar{f}"'] \arrow[dd,tail]
& & \eta_G(P_G) \arrow[dl,dotted,"\bar{g}"] \arrow[dd,tail,"{i_{ab(G)}}"] \\
 & P_A \arrow[dd,tail,near start,"i_A"] & \\
G \arrow[rr,two heads,near start,"\eta_G"] \arrow[dr,"f"'] 
& & ab(G). \arrow[dl,dotted,"g"] \\
 & A & 
\end{tikzcd}
\end{center}
Then, the universal property of the abelianization of $G$ yields a unique arrow $g \colon ab(G) \rightarrow A$ such that $g \cdot \eta_G = f$ in the category $\mathsf{Grp}$. Knowing that $\bar{\eta}_G$ is a strong epimorphism in the category $\mathsf{Mon}$ of monoids, there is moreover a unique monoid morphism $\bar{g} \colon \eta_G(P_G) \rightarrow P_A$ making the following diagram commute:
\begin{center}
\begin{tikzcd}
P_G \arrow[r,two heads,"\bar{\eta}_G"] \arrow[d,"\bar{f}"']
& \eta_G(P_G) \arrow[d,"{g \cdot i_{ab(G)}}"] \arrow[dl,dotted,"\bar{g}"] \\
P_A \arrow[r,tail,"i_A"']
& A.
\end{tikzcd}
\end{center}
Accordingly, there exists a unique morphism $(g,\bar{g}) \colon (ab(G),\eta_G(P_G)) \rightarrow (A,P_A)$ in $\mathsf{PreOrdGrp}$ such that $(g,\bar{g}) \cdot (\eta_G,\bar{\eta}_G) = (f,\bar{f})$. 
\end{proof}

\begin{corollary}
If $\E_C$ denotes the class of normal epimorphisms in $\mathsf{PreOrdGrp}$ and $\z_C$ the class of normal epimorphisms in $\mathsf{PreOrdAb}$, then 
\begin{equation} \label{Galois structure}
\Gamma_C = (\mathsf{PreOrdGrp}, \mathsf{PreOrdAb}, C, V, \E_C, \z_C)
\end{equation}
is a Galois structure.
\end{corollary}

We will now prove that this Galois structure is actually admissible. In order to do this, we will use the following result which is a reformulation of a result due to Janelidze and Kelly \cite{JK}:

\begin{proposition} \label{prop JK}
Let $\C$ be a normal category. Let $\Gamma = (\C,\F,F,U,\E,\z)$ be a Galois structure, where $\E$ is the class of normal epimorphisms in $\C$ and $\F$ a full reflective subcategory of $\C$ closed under subobjects and quotients. If the lattice $\mathsf{Norm(\C)}$ of normal subobjects in $\C$ is modular, then the Galois structure $\Gamma$ is admissible.
\end{proposition}

\begin{proof}
Consider any pullback in $\C$ of the following form
\begin{equation} \label{diagram1JK}
\begin{tikzcd}
P \arrow[r,two heads,"\pi_2"] \arrow[d,two heads,"\pi_1"']
& A \arrow[d,two heads,"\phi"] \\
B \arrow[r,two heads,"\eta_B"']
& F(B)
\end{tikzcd}
\end{equation}
where $\eta_B$ is the $B$-component of the unit of the adjunction $F \dashv U$, $A$ an object of the subcategory $\F$ and $\phi$ any morphism in the class $\E$. We write $s \colon S \rightarrowtail P$ and $t \colon T \rightarrowtail P$ for the kernels of $\pi_1$ and $\pi_2$, respectively. We need to prove that this pullback is preserved by the reflector $F$. In order to do this, we consider the next commutative diagram
\begin{equation} \label{diagram2JK}
\begin{tikzcd} [column sep=large,row sep=large]
R \vee S \arrow[dr,tail]
& S \arrow[d,tail,"s"] \arrow[l,tail]
& & \\
R \arrow[u,tail] \arrow[r,tail,"r"] \arrow[d,dotted,tail,"m"]
& P \arrow[r,two heads,"\eta_P"] \arrow[rr,bend left,two heads,"\pi_2"] \arrow[dr,two heads,"q"] \arrow[d,two heads,"\pi_1"']
& F(P) \arrow[r,two heads,"\psi"] \arrow[d,two heads,"F(\pi_1)"]
& A \arrow[d,two heads,"\phi"] \\
T \arrow[uu,bend left,dotted,tail,"n"] \arrow[ur,tail,"t"']
& B \arrow[r,two heads,"\eta_B"'] 
& F(B) \arrow[r,equal]
& F(B)
\end{tikzcd}
\end{equation}
where $r \colon R \rightarrowtail P$ is the kernel of the $P$-component $\eta_P$ of the unit of the adjunction $F \dashv U$, $R \vee S$ is the supremum of $R$ and $S$ in $\mathsf{Norm(\C)}$, $q$ is the composite $\eta_B \cdot \pi_1 = F(\pi_1) \cdot \eta_P$, and $\psi$ is the unique morphism induced by the universal property of $\eta_P$ such that $\psi \cdot \eta_P = \pi_2$. We compute that $\pi_2 \cdot r = \psi \cdot \eta_P \cdot r = 0$ so that, by the universal property of kernels, there exists a unique arrow $m \colon R \rightarrow T$ such that $t \cdot m = r$, hence $R \le T$. The assumption that $\F$ is a full reflective subcategory of $\C$ closed under subobjects and quotients implies that the middle square of the above diagram \eqref{diagram2JK} is a pushout. As a consequence, the supremum $R \vee S$ exists, and it is 
the kernel of the morphism $q$. Accordingly, since $q \cdot t = F(\pi_1) \cdot \eta_P \cdot t = \phi \cdot \psi \cdot \eta_P \cdot t = \phi \cdot \pi_2 \cdot t = 0$, there exists a unique morphism $n \colon T \rightarrow R \vee S$ such that $\ker q \cdot n = t$, so that $T \le R \vee S$ in $\mathsf{Norm(\C)}$. Observe then that
\[R = R \vee \{0\} = R \vee (S \wedge T) = (R \vee S) \wedge T = T,\]
where the second equality follows from the fact that $\pi_1$ and $\pi_2$ are jointly monomorphic, the third equality by modularity of $\mathsf{Norm(\C)}$ (since $R \le T$), and the last one from the fact that $T \le R \vee S$. This means that $\ker(\eta_P) = \ker(\pi_2)$,  and the induced morphism $\psi$ is then an isomorphism. If we apply the reflector $F$ to the pullback \eqref{diagram1JK}, we then get a pullback in $\F$, and the Galois structure $\Gamma$ is therefore admissible, as desired. 
\end{proof}

We will now apply this Proposition to the adjunction we are interested in. It just remains to prove that the lattice $\mathsf{Norm(PreOrdGrp)}$ of normal subobjects in $\mathsf{PreOrdGrp}$ is modular. In order to do this, we first need to describe the supremum of two normal subobjects in $\mathsf{PreOrdGrp}$.

\begin{lemma}
Let $(A,P_A)$ and $(B,P_B)$ be two normal subobjects in $\mathsf{PreOrdGrp}$ of a preordered group $(G,P_G)$. The supremum of $(A,P_A)$ and $(B,P_B)$ in $\mathsf{Norm(PreOrdGrp)}$ is given by
\[(A,P_A) \vee (B,P_B) = (A \cdot B, (A \cdot B) \cap P_G)\]
where $A \cdot B = \{a + b \mid a \in A \ \text{and} \ b \in B\}$ is the supremum of the normal subgroups $A$ and $B$ in the category $\mathsf{Grp}$ of groups.
\end{lemma}

\begin{proof}
We first note that $(A \cdot B,(A \cdot B) \cap P_G)$ is a normal subobject of $(G,P_G)$. Indeed, $A \cdot B$ is a normal subgroup of $G$ and the square
\begin{center}
\begin{tikzcd}
(A \cdot B) \cap P_G \arrow[r,tail] \arrow[d,tail]
& P_G \arrow[d,tail] \\
A \cdot B \arrow[r,tail]
& G
\end{tikzcd}
\end{center}
is a pullback in the category $\mathsf{Mon}$ of monoids. It is also clear that $(A,P_A) \le (A \cdot B,(A \cdot B) \cap P_G)$ and that $(B,P_B) \le (A \cdot B,(A \cdot B) \cap P_G)$ (since $P_A = A \cap P_G$ and $P_B = B \cap P_G$). Consider next that we have another normal subobject $(N,P_N)$ of $(G,P_G)$ such that $(A,P_A) \le (N,P_N)$ and $(B,P_B) \le (N,P_N)$. Then of course $A \cdot B \le N$ since $A \cdot B = A \vee B$ in $\mathsf{Grp}$. As a consequence, $(A \cdot B,(A \cdot B) \cap P_G) \le (N,P_N)$. 
\end{proof}

\begin{proposition} \label{Modularity of PreOrdGrp}
The lattice $\mathsf{Norm(PreOrdGrp)}$ of normal subobjects in the category $\mathsf{PreOrdGrp}$ of preordered groups is modular: for any triple $(A,P_A)$, $(B,P_B)$ and $(C,P_C)$ of normal subobjects in $\mathsf{PreOrdGrp}$ of a given preordered group $(G,P_G)$, such that $(C,P_C) \le (A,P_A)$, we have that
\[(A,P_A) \wedge ((B,P_B) \vee (C,P_C)) = ((A,P_A) \wedge (B,P_B)) \vee (C,P_C).\]
\end{proposition}

\begin{proof}
We already know that $A \wedge (B \vee C) = (A \wedge B) \vee C$ since the lattice $\mathsf{Norm(Grp)}$ of normal subgroups of $G$ is modular. We therefore compute that
\begin{align*}
(A,P_A) \wedge ((B,P_B) \vee (C,P_C)) & = (A,P_A) \wedge (B \cdot C, (B \cdot C) \cap P_G)\\
& = (A \cap (B \cdot C),P_A \cap (B \cdot C) \cap P_G)\\
& = (A \cap (B \cdot C),A \cap P_G \cap (B \cdot C) \cap P_G)\\
& = (A \cap (B \cdot C),(A \cap (B \cdot C)) \cap P_G)\\
& = (A \wedge (B \vee C),(A \wedge (B \vee C)) \cap P_G)\\
& = ((A \wedge B) \vee C,((A \wedge B) \vee C) \cap P_G)\\
& = (A \cap B) \cdot C,((A \cap B) \cdot C) \cap P_G)\\
& = (A \cap B,P_A \cap P_B) \vee (C,P_C)\\
& = ((A,P_A) \wedge (B,P_B)) \vee (C,P_C),
\end{align*}
where we have used the fact that the infimum of two normal subobjects in $\mathsf{Norm(PreOrdGrp)}$ is given by their pullback.
\end{proof}

\begin{corollary}
The Galois structure \eqref{Galois structure} is admissible.
\end{corollary}

\begin{proof}
This is a direct consequence of Propositions \ref{prop JK} and \ref{Modularity of PreOrdGrp}.
\end{proof}

%Let us now recall the following well-known property:
%
%\begin{lemma}
%Consider, in any finitely complete category, the following commutative diagram
%\begin{center}
%\begin{tikzcd}
%Eq(f') \arrow[r,shift left,"f'_1"] \arrow[r,shift right,"f'_2"'] \arrow[d]
%& A' \arrow[r,"f'"] \arrow[d,tail,"i"]
%& B' \arrow[d,tail,"j"] \\
%Eq(f) \arrow[r,shift left,"f_1"] \arrow[r,shift right,"f_2"']
%& A \arrow[r,"f"']
%& B.
%\end{tikzcd}
%\end{center}
%If the arrows $i$ and $j$ are monomorphisms, then
%\[Eq(f') = i^{-1}(Eq(f)) = Eq(f) \cap (A' \times A').\]
%\end{lemma}
%
%\begin{corollary} \label{Eq(bar(f))}
%Consider, in the category $\mathsf{PreOrdGrp}$ of preordered groups, any morphism $$(f,\bar{f}) \colon (G,P_G) \rightarrow (H,P_H)$$ and its kernel pair $Eq(f,\bar{f}) = (Eq(f),Eq(\bar{f}))$:
%\begin{center}
%\begin{tikzcd}
%Eq(\bar{f}) \arrow[r,shift left] \arrow[r,shift right] \arrow[d,tail]
%& P_G \arrow[r,"\bar{f}"] \arrow[d,tail]
%& P_H \arrow[d,tail] \\
%Eq(f) \arrow[r,shift left] \arrow[r,shift right]
%& G \arrow[r,"f"']
%& H.
%\end{tikzcd}
%\end{center}
%Then $Eq(\bar{f}) = Eq(f) \cap (P_G \times P_G)$. In other words, the square
%\begin{center}
%\begin{tikzcd}
%Eq(\bar{f}) \arrow[r,tail] \arrow[d,tail]
%& P_G \times P_G \arrow[d,tail]\\
%Eq(f) \arrow[r,tail]
%& G \times G
%\end{tikzcd}
%\end{center}
%is a pullback in the category $\mathsf{Mon}$ of monoids.
%\end{corollary}

We will prove in Theorem \ref{GammaC-normal extensions} that the $\Gamma_C$-normal extensions are given by the normal epimorphisms
\begin{tikzcd}
(f,\bar{f}) \colon (G,P_G) \arrow[r,two heads]
& (H,P_H)
\end{tikzcd} 
such that
\begin{itemize}
\item[(i)] 
\begin{tikzcd}
f \colon G \arrow[r,two heads]
& H
\end{tikzcd} 
is an \emph{algebraically central extension}, i.e. $\mathsf{Ker}(f) \subseteq Z(G)$;
\item[(ii)] the following Condition $(\star)$ holds for $(f,\bar{f})$: 
\vspace{3mm}

\centerline{$(\star) \qquad$ for any $(a,b,c) \in Eq(\bar{\eta}_G) \times_{P_{G}} Eq(\bar{f})$, $a - b + c \in P_G$,}
\vspace{3mm}
\noindent where $\bar{\eta}_G \colon P_G \twoheadrightarrow \eta_G(P_G)$ is the restriction of $\eta_G \colon G \twoheadrightarrow G/[G,G]$ to the positive cones.
\end{itemize}

The following two lemmas will be helpful to establish this characterization.

\begin{lemma} \label{stability of Condition star}
Consider the following pullback
\begin{center}
\begin{tikzcd}[row sep=large,column sep=large]
(P,P_P) \arrow[r,two heads,"{(p_2,\bar{p}_2)}"] \arrow[d,two heads,"{(p_1,\bar{p}_1)}"']
& (G,P_G) \arrow[d,two heads,"{(f,\bar{f})}"]\\
(E,P_E) \arrow[r,two heads,"{(p,\bar{p})}"']
& (H,P_H)
\end{tikzcd}
\end{center}
 in $\mathsf{PreOrdGrp}$, where all the arrows are regular epimorphisms.
\begin{enumerate}
\item If Condition $(\star)$ holds for $(f,\bar{f})$, then it holds for $(p_1,\bar{p}_1)$.
\item If $(p,\bar{p}) = (f,\bar{f})$, then Condition $(\star)$ holds for $(p_1,\bar{p}_1)$ if and only if it holds for $(f,\bar{f})$.
\end{enumerate}
\end{lemma}

\begin{proof}
\begin{enumerate}
\item Let $((e_1,x_1),(e_2,x_2),(e_3,x_3)) \in Eq(\bar{\eta}_P) \times_{P_P} Eq(\bar{p}_1)$. Then, this means that
\begin{itemize}
\item $e_i \in P_E$ and $x_i \in P_G$ for any $i = 1, 2, 3$;
\item $p(e_i) = f(x_i)$ for any $i = 1, 2, 3$;
\item $\eta_P(e_1,x_1) = \eta_P(e_2,x_2)$, which implies that
\[(e_1 - e_2,x_1 - x_2) \in \mathsf{Ker}(\eta_P) = [P,P],\]
hence, in particular, $x_1 - x_2 \in [G,G]$;
\item $p_1(e_2,x_2) = p_1(e_3,x_3)$, i.e. $e_2 = e_3$.
\end{itemize}
As a consequence, 
\begin{itemize}
\item $\eta_G(x_1) = \eta_G(x_2)$, i.e. $(x_1,x_2) \in Eq(\bar{\eta}_G)$;
\item $f(x_2) = p(e_2) = p(e_3) = f(x_3)$, i.e. $(x_2,x_3) \in Eq(\bar{f})$.
\end{itemize}
In other words, $(x_1,x_2,x_3) \in Eq(\bar{\eta}_G) \times_{P_G} Eq(\bar{f})$. By assumption the element $x_1 - x_2 + x_3$ then belongs to $P_G$, and one has that
\[(e_1,x_1) - (e_2,x_2) + (e_3,x_3) = (e_1 - e_2 + e_3,x_1 - x_2 + x_3) = (e_1, x_1 - x_2 + x_3) \in P_P,\]
as desired. 
\item Thanks to $(1)$, it suffices to check that, if Condition $(\star)$ holds for $(p_1,\bar{p}_1)$ (with $(p_1,\bar{p}_1)$ the first projection of the kernel pair of $(f,\bar{f})$), then it also holds for $(f,\bar{f})$. \\ Let $(a,b,c) \in Eq(\bar{\eta}_G) \times_{P_G} Eq(\bar{f})$. In particular, this means that $a, b, c \in P_G$,  $a-b \in \mathsf{Ker} (\eta_G)  = [G,G]$, and $f(b) = f(c)$. The fact that $a-b \in [G,G]$ implies that 
$$(a,a) - (b,b) = (a-b,a-b) = ([x_1, x_2] + [x_3,x_4] + \cdots + [x_{n-1}, x_n], [x_1, x_2] + [x_3,x_4] + \cdots + [x_{n-1}, x_n]),$$
for some suitable elements $x_i \in G$ (with $i \in \{1, \cdots , n\}$).
It follows that
\[(a,a) - (b,b) = (x_1,x_1) + (x_2,x_2) - (x_1,x_1) - (x_2,x_2)  + \cdots +  (x_{n-1},x_{n-1}) + (x_n,x_n) - (x_{n-1},x_{n-1}) - (x_n,x_n) 
,\]
so that $(a,a) - (b,b) \in [Eq(f),Eq(f)] = \mathsf{Ker}(\eta_{Eq(f)})$.
This means that $((a,a),(b,b)) \in Eq(\bar{\eta}_{Eq(f)})$. 
On the other hand $p_1(b,b) = p_1(b,c)$, and this implies that $((b,b),(b,c)) \in Eq(\bar{p}_1)$, and then 
$$((a,a),(b,b),(b,c)) \in Eq(\bar{\eta}_{Eq(f)}) \times_{Eq(\bar{f})} Eq(\bar{p}_1).$$ So, by assumption,
\[(a,a-b+c) = (a,a) - (b,b) + (b,c) \in Eq(\bar{f}),\]
hence, in particular, $a-b+c \in P_G$. 
\qedhere
\end{enumerate}
\end{proof}

\begin{lemma} \label{Condition star in PreOrdAb}
Condition $(\star)$ holds for any morphism $(f,\bar{f}) \colon (G,P_G) \rightarrow (H,P_H)$ in $\mathsf{PreOrdGrp}$ where $(G,P_G) \in \mathsf{PreOrdAb}$.
\end{lemma}

\begin{proof}
Let $(a,b,c) \in Eq(\bar{\eta}_G) \times_{P_G} Eq(\bar{f})$. Since $G \in \mathsf{Ab}$, we have that $\eta_G = 1_G$, and then that $a = \eta_G(a) = \eta_G(b) = b$. It follows that $a - b + c = c \in P_G$.
\end{proof}

\begin{theorem} \label{GammaC-normal extensions}
Let 
\begin{tikzcd}
(f,\bar{f}) \colon (G,P_G) \arrow[r,two heads]
& (H,P_H)
\end{tikzcd}
be a regular epimorphism in $\mathsf{PreOrdGrp}$. The following conditions are equivalent:
\begin{enumerate}
\item
\begin{itemize} 
\item[(i)] $\mathsf{Ker}(f) \subseteq Z(G)$; 
\item[(ii)] for any $(a,b,c) \in Eq(\bar{\eta}_G) \times_{P_{G}} Eq(\bar{f})$, we have $a - b + c \in P_G$.
\end{itemize}
\item $(f,\bar{f})$ is a $(\Gamma_C\text{-})$normal extension.
\end{enumerate}
\end{theorem}

\begin{proof}
\begin{itemize}
\item $(1) \Rightarrow (2)$: We need to prove that the first projection $(\pi_1,\bar{\pi}_1)$ of the kernel pair of $(f,\bar{f})$ is a ($\Gamma_C$-)trivial extension, i.e. that the square below
\begin{equation}\label{naturality square for C}
\begin{tikzcd}[row sep=large,column sep=large]
{Eq(f,\bar{f})} \arrow[rr,two heads,"{(\eta_{Eq(f)},\bar{\eta}_{Eq(f)})}"] \arrow[dd,two heads,"{(\pi_1,\bar{\pi}_1)}"']
&
& {C(Eq(f,\bar{f}))} \arrow[dd,two heads,"{C(\pi_1,\bar{\pi}_1) = (C(\pi_1),C(\bar{\pi}_1))}"]\\
 & & \\
{(G,P_G)} \arrow[rr,two heads,"{(\eta_G,\bar{\eta}_G)}"']
&
& {C(G,P_G)}
\end{tikzcd}
\end{equation}
is a pullback in the category $\mathsf{PreOrdGrp}$ of preordered groups. First of all it is well-known that (i) implies that its restriction to the category $\mathsf{Grp}$ of groups is a pullback (in $\mathsf{Grp}$) \cite{JG90}. It remains to show that the external square in the diagram
\begin{center}
\begin{tikzcd}[row sep=large,column sep=large]
Eq(\bar{f}) \arrow[rr,two heads,"{\bar{\eta}_{Eq(f)}}"] \arrow[dr,tail,dotted,"\bar{\phi}"'] \arrow[dd,two heads,"{\bar{\pi}_1}"']
& & \eta_{Eq(f)}(Eq(\bar{f})) \arrow[dd,two heads,"{C(\bar{\pi}_1)}"]\\
 & P_P \arrow[ur,two heads,"{\bar{q}_2}"'] \arrow[dl,two heads,"{\bar{q}_1}"]
& \\
P_G \arrow[rr,two heads,"{\bar{\eta}_G}"']
& & \eta_G(P_G)
\end{tikzcd}
\end{center}
is a pullback in $\mathsf{Mon}$. Consider then the pullback $(P,P_P)$ of $(\eta_G,\bar{\eta}_G)$ and $C(\pi_1,\bar{\pi}_1)$ (denoted with a slight abuse of notation by $(C(\pi_1),C(\bar{\pi}_1)$)) in $\mathsf{PreOrdGrp}$ (with the two projections $(q_1,\bar{q}_1)$ and $(q_2,\bar{q}_2)$), as well as the induced morphism $(\phi,\bar{\phi}) \colon (Eq(f), Eq(\bar{f})) \rightarrow (P,P_P)$ to the pullback (with $\phi$ an isomorphism, as already observed). Since $\phi$ is an isomorphism, obviously $\bar{\phi}$ is a monomorphism by commutativity of the following square:
\begin{center}
\begin{tikzcd}
Eq(\bar{f}) \arrow[r,tail,"{\bar{\phi}}"] \arrow[d,tail]
& P_P \arrow[d,tail]\\
Eq(f) \arrow[r,"\phi"',"{\cong}"]
& P.
\end{tikzcd}
\end{center}
So it suffices to show that $\bar{\phi}$ is surjective. Let $(a,\eta_{Eq(f)}(b,c)) \in P_P$. In other words, $a,b,c \in P_G$ are such that $f(b) = f(c)$ and $\eta_G(a) = C(\bar{\pi}_1)(\eta_{Eq(f)}(b,c)) = (\eta_G \cdot \pi_1)(b,c) = \eta_G(b)$, that is, $(a,b,c) \in Eq(\bar{\eta}_G) \times_{P_G} Eq(\bar{f})$. By (ii), we then have that $(a,a-b+c) \in Eq(\bar{f})$. We are going to show that this element is sent by $\phi$ to $(a,\eta_{Eq(f)}(b,c))$. We first observe that, since $(a,b) \in Eq(\bar{\eta}_G)$, there exist $x_i \in G$ (for $i \in \{1,\cdots,n\}$) such that $$a - b = x_1 + x_2 - x_1 - x_2 + x_3 + x_4 - x_3 - x_4 + \cdots + x_{n-1} + x_n - x_{n-1} - x_n.$$
This implies that
\begin{align*}
(a-b,a-b) = & (x_1,x_1) + (x_2,x_2) - (x_1,x_1) - (x_2,x_2) \\
&  + \cdots + (x_{n-1},x_{n-1}) + (x_n,x_n) - (x_{n-1},x_{n-1}) - (x_n,x_n),
\end{align*}
i.e. $(a-b,a-b) \in [Eq(f),Eq(f)] = \mathsf{Ker}(\eta_{Eq(f)})$. As a consequence we can compute
\begin{align*}
\phi(a,a-b+c) & = (a,\eta_{Eq(f)}(a,a-b+c))\\
& = (a,\eta_{Eq(f)}(a-b,a-b) + \eta_{Eq(f)}(b,c))\\
& = (a,\eta_{Eq(f)}(b,c)),
\end{align*}
and deduce that the homomorphism $\bar{\phi}$ is then surjective, hence an isomorphism. This proves that the square \eqref{naturality square for C} is a pullback in $\mathsf{PreOrdGrp}$, i.e. that $(f,\bar{f})$ is a ($\Gamma_C$-)normal extension.
\item $(2) \Rightarrow (1)$: Since $(f,\bar{f})$ is a $(\Gamma_C\text{-})$normal extension, the two squares below are then pullbacks in $\mathsf{PreOrdGrp}$:
\begin{center}
\begin{tikzcd}[row sep=large,column sep=large]
{(Eq(f),Eq(\bar{f}))} \arrow[r,two heads,"{(\pi_2,\bar{\pi}_2)}"] \arrow[d,two heads,"{(\pi_1,\bar{\pi}_1)}"']
& (G,P_G) \arrow[d,two heads,"{(f,\bar{f})}"]\\
(G,P_G) \arrow[r,two heads,"{(f,\bar{f})}"']
& (H,P_H)
\end{tikzcd}
\end{center}
\begin{center}
\begin{tikzcd}[row sep=large,column sep=large]
{(Eq(f),Eq(\bar{f}))} \arrow[rr,two heads,"{(\eta_{Eq(f)},\bar{\eta}_{Eq(f)})}"] \arrow[dd,two heads,"{(\pi_1,\bar{\pi}_1)}"']
&
& {(ab(Eq(f)),\eta_{Eq(f)}(Eq(\bar{f})))} \arrow[dd,two heads,"{C(\pi_1,\bar{\pi}_1)}"]\\
 & & \\
(G,P_G) \arrow[rr,two heads,"{(\eta_G,\bar{\eta}_G)}"']
&
& {(ab(G),\eta_G(P_G)).}
\end{tikzcd}
\end{center}
In this second diagram, $C(\pi_1,\bar{\pi}_1) \in \mathsf{PreOrdAb}$. So, thanks to Lemma \ref{Condition star in PreOrdAb}, we know that $C(\pi_1,\bar{\pi}_1)$ satisfies Condition $(\star)$. Now, using Lemma \ref{stability of Condition star}(1) with the second pullback, we get that Condition $(\star)$ holds for $(\pi_1,\bar{\pi}_1)$. The application of Lemma \ref{stability of Condition star}(2) to the first pullback next gives the validity of Condition $(\star)$ for $(f,\bar{f})$, which corresponds to (ii). Condition (i) follows directly from the fact that 
\begin{tikzcd}
f \colon G \arrow[r,two heads]
& H
\end{tikzcd}
is a normal extension with respect to the admissible Galois structure induced by the abelianization functor $\mathsf{ab} \colon \mathsf{Grp} \rightarrow \mathsf{Ab}$ (see \cite{JG90, JK}). \qedhere
\end{itemize}
\end{proof}

%Commutative and abelian objects
\section{Commutative and abelian objects in the category of preordered groups} \label{Commutative and abelian objects}

%The characterization of ($\Gamma_C$-)central extensions does not seem to be as “beautiful” as the one of ($\Gamma_C$-)normal extensions, in the sense that they %do not seem to have a purely algebraic description. Moreover, we are not even sure that both notions coincide! A reason that might explain why it does not work very %well is the fact that the subcategory $\mathsf{PreOrdAb}$ of preordered abelian groups is not the full subcategory of \emph{abelian objects} in $\mathsf{PreOrdGrp}$ %(in the sense of \cite{BB}) as it is the case for $\mathsf{Ab}$ with respect to $\mathsf{Grp}$, but only the full subcategory of \emph{commutative objects} (in the sense %of \cite{BB}). This is why we consider in the next sections the functor (from $\mathsf{PreOrdGrp}$) to the full subcategory of abelian objects in $\mathsf{PreOrdGrp}$, %which turns out to be the subcategory of preordered abelian groups with a group as a positive cone (or, equivalently, with an equivalence relation as a preorder). 

This section is devoted to the characterization of the commutative and the abelian objects in the category of preordered groups. For the reader's convenience, we first recall some definitions and results (we refer to \cite{BB} for more details).

A pointed category $\C$ with finite limits is said to be \emph{unital} when, for any objects $X, Y \in \C$, the pair $(l_X, r_Y)$ in the diagram
\begin{center}
\begin{tikzcd}
X \arrow[r,"l_X"]
& X \times Y 
& Y \arrow[l,"r_Y"']
\end{tikzcd}
\end{center}
is strongly epimorphic, where $l_X = \langle 1_X,0 \rangle$ and $r_Y = \langle 0, 1_Y \rangle$. 

\begin{definition} 
Let $\C$ be a unital category. An object $X \in \C$ is said to be \emph{commutative} when there exists a morphism $\phi \colon X \times X \rightarrow X$ making the following diagram commute:
\begin{center}
\begin{tikzcd}
X \arrow[r,"l_X"] \arrow[dr,equal]
& X \times X \arrow[d,dotted,"\phi"]
& X \arrow[l,"r_X"'] \arrow[dl,equal]\\
 & X. &
\end{tikzcd}
\end{center}
\end{definition}

Note that the morphism $\phi$ above is necessarily unique by unitality of the category $\C$. It turns out that any commutative object in a unital category is an (internal) commutative monoid. 

\begin{proposition} \label{ComM(C)} \cite{BB}
Let $\C$ be a unital category. The category $\mathsf{ComMon}(\C)$ of \emph{internal commutative monoids} in $\C$ is the full subcategory of commutative objects in $\C$.
\end{proposition}

\begin{definition} \label{definition abelian object} \cite{BB}
Let $\C$ be a unital category. An object $X \in \C$ is said to be \emph{abelian} when it is commutative and the corresponding internal commutative monoid is an internal abelian group.
\end{definition}

The following Proposition gives a useful characterization of abelian objects in a unital category:

\begin{proposition} \label{characterization of abelian objects} \cite{BB}
Let $\C$ be a unital category. Then an object $X$ of $\C$ is abelian if and only if there exists a morphism $\phi \colon X \times X \rightarrow X$ making the diagram 
\begin{center}
\begin{tikzcd}
X \arrow[r,"r_X"] \arrow[dr,equal]
& X \times X \arrow[d,dotted,"\phi"]
& X \arrow[l,"{\Delta_X}"'] \arrow[dl,"0"]\\
 & X &
\end{tikzcd}
\end{center}
commute, where $\Delta_X = \langle 1_X,1_X \rangle$ is the diagonal of $X$.
\end{proposition}

A pointed category $\C$ with finite limits is said to be \emph{strongly unital} when, for any object $X \in \C$, the pair $(r_X,\Delta_X)$ in the diagram 
\begin{center}
\begin{tikzcd}
X \arrow[r,"r_X"]
& X \times X
& X \arrow[l,"\Delta_X"']
\end{tikzcd}
\end{center}
is strongly epimorphic, where $r_X = \langle 0, 1_X \rangle$ and $\Delta_X = \langle 1_X,1_X \rangle$.

Any abelian object is by definition commutative, but the converse is not true in general. In the strongly unital context though, the converse holds:

\begin{proposition} \label{commutative objects in strongly unital categories} \cite{BB}
In a strongly unital category, any commutative object is an abelian object.
\end{proposition}

As proved in \cite{CMFM}, the category $\mathsf{PreOrdGrp}$ of preordered groups is a unital category, so that the notions of commutative and abelian objects both make sense in this setting. It turns out that the commutative objects in $\mathsf{PreOrdGrp}$ are precisely the preordered abelian groups:

\begin{proposition} \label{commutative objects in PreOrdGrp}
The category $\mathsf{PreOrdAb}$ of preordered abelian groups coincides with the category $\mathsf{ComMon}(\mathsf{PreOrdGrp})$ of internal commutative monoids in $\mathsf{PreOrdGrp}$:
\[\mathsf{PreOrdAb} = \mathsf{ComMon}(\mathsf{PreOrdGrp}).\]
\end{proposition}

\begin{proof}
Let $(G,P_G)$ be a commutative object in $\mathsf{PreOrdGrp}$. Then, this means that there exists a morphism $(\phi,\bar{\phi}) \colon (G,P_G) \times (G,P_G) \rightarrow (G,P_G)$ in $\mathsf{PreOrdGrp}$ such that the diagram
\begin{center}
\begin{tikzcd}[row sep=large,column sep=large]
(G,P_G) \arrow[r,"{(l_G,\bar{l}_G)}"] \arrow[dr,equal]
& (G \times G,P_G \times P_G) \arrow[d,dotted,"{(\phi,\bar{\phi})}"]
& (G,P_G) \arrow[l,"{(r_G,\bar{r}_G)}"'] \arrow[dl,equal]\\
& (G,P_G) &
\end{tikzcd}
\end{center}
commutes. In particular, for any $x \in G$ and any $y \in G$,
\begin{align*}
x + y & = (\phi \cdot l_G)(x) + (\phi \cdot r_G)(y)\\
& = \phi(x,0) + \phi(0,y) \\ &= \phi((x,0) + (0,y))\\
& = \phi(x,y)\\
& = \phi((0,y) + (x,0)) \\ &= \phi(0,y) + \phi(x,0)\\
& = (\phi \cdot r_G)(y) + (\phi \cdot l_G)(x)\\
& = y + x,
\end{align*}
which means that $G$ is an abelian group and then that $(G,P_G) \in \mathsf{PreOrdAb}$.

Conversely, consider $(G,P_G) \in \mathsf{PreOrdAb}$ and define, for any $(x,y) \in G \times G$, the requested morphism $\phi \colon G \times G \rightarrow G$ as follows: $\phi(x,y) = x + y$. Then it is easy to check that $\phi$ is a group morphism
%Indeed, 
%\begin{itemize}
%\item $\phi(0,0) = 0 + 0 = 0$;
%\item for any $(x,y), (x',y') \in G \times G$,
%\begin{align*}
%\phi((x,y) + (x',y')) & = \phi(x + x',y + y') = (x + x') + (y + y')\\
%& = x + y + x' + y'\\
%& = \phi(x,y) + \phi(x',y')
%\end{align*}
%since the group $G$ is abelian.
%\end{itemize}
%We now check that 
and that $\phi \cdot l_G = 1_G$ and $\phi \cdot r_G = 1_G$.
Moreover, it is clear that the restriction $\bar{\phi}$ of $\phi$ to $P_G \times P_G$ takes its values in $P_G$, since $P_G$ is a submonoid of $G$. As a consequence, $(\phi,\bar{\phi}) \colon (G,P_G) \times (G,P_G) \rightarrow (G,P_G)$ is a morphism in $\mathsf{PreOrdGrp}$ making the above diagram commute in $\mathsf{PreOrdGrp}$. This means that $(G,P_G)$ is a commutative object in $\mathsf{PreOrdGrp}$. The result then follows from Proposition \ref{ComM(C)}.
\end{proof}

\begin{proposition} \label{abelian objects in PreOrdGrp}
The abelian objects in $\mathsf{PreOrdGrp}$ are the preordered groups $(G,P_G)$ where $G$ is an abelian group and $P_G$ a (normal) subgroup of $G$.
\end{proposition}

\begin{proof}
Let $(G,P_G)$ be an abelian object in $\mathsf{PreOrdGrp}$. In particular, it is a commutative object, so that $G \in \mathsf{Ab}$ (by Proposition \ref{commutative objects in PreOrdGrp}). Thanks to Proposition \ref{characterization of abelian objects}, we also have that there exists a morphism $(\phi,\bar{\phi}) \colon (G,P_G) \times (G,P_G) \rightarrow (G,P_G)$ making the following diagram commute:
\begin{center}
\begin{tikzcd}[row sep=large,column sep=large]
(G,P_G) \arrow[r,"{(r_G,\bar{r}_G)}"] \arrow[dr,equal]
& (G \times G,P_G \times P_G) \arrow[d,dotted,"{(\phi,\bar{\phi})}"]
& (G,P_G) \arrow[l,"{(\Delta_G,\bar{\Delta}_G)}"'] \arrow[dl,"{(0,0)}"]\\
& (G,P_G). &
\end{tikzcd}
\end{center}
The restriction $\bar{\phi}$ of $\phi$ to $P_G \times P_G$ takes its values in $P_G$. Accordingly, for any $x \in P_G$, $\phi(x,0) \in P_G$. Now we compute that
\[\phi(x,0) = \phi((0,-x) + (x,x)) = \phi(0,-x) + \phi(x,x) = (\phi \cdot r_G)(-x) + (\phi \cdot \Delta_G)(x) = -x + 0 = -x.\]
As a consequence, $-x \in P_G$ for any $x \in P_G$, which proves that $P_G$ is a subgroup of $G$, as desired.

Conversely, consider any preordered abelian group $(G,P_G)$ where the positive cone $P_G$ is a subgroup of $G$. Define the morphism $\phi \colon G \times G \rightarrow G$ by $\phi(x,y) = -x + y$, for any $(x,y) \in G \times G$. It is easily seen that $\phi$ is a group morphism since the group $G$ is abelian. Moreover, for any $x \in G$,
\begin{itemize}
\item $(\phi \cdot r_G)(x) = \phi(0,x) = -0 + x = x$, so that $\phi \cdot r_G = 1_G$;
\item $(\phi \cdot \Delta_G)(x) = \phi(x,x) = -x + x = 0$, so that $\phi \cdot \Delta_G = 0$.
\end{itemize}
Let us now prove that the restriction $\bar{\phi}$ of $\phi$ to $P_G \times P_G$ takes its values in $P_G$. Let $(x,y) \in P_G \times P_G$. Since $P_G$ is a subgroup of $G$, then $-x \in P_G$ and then $-x + y =\phi(x,y) \in P_G$. Accordingly, the pair $(\phi,\bar{\phi}) \colon (G,P_G) \times (G,P_G) \rightarrow (G,P_G)$ is a morphism in $\mathsf{PreOrdGrp}$ making the diagram above commute. By Proposition \ref{characterization of abelian objects}, we conclude that $(G,P_G)$ is an abelian object in $\mathsf{PreOrdGrp}$.
\end{proof}

\begin{corollary}\label{AbelianObjects}
The category $\mathsf{Ab(PreOrdGrp)}$ of internal abelian groups in $\mathsf{PreOrdGrp}$ is isomorphic to the category $\mathsf{Mono}(\mathsf{Ab})$ of monomorphisms in the category $\mathsf{Ab}$ of abelian groups:
$$\mathsf{Ab(PreOrdGrp)}= \mathsf{Mono}(\mathsf{Ab}).$$
\end{corollary}

\begin{proof}
This is a direct consequence of Proposition \ref{abelian objects in PreOrdGrp} and Definition \ref{definition abelian object}.
\end{proof}

Another consequence of Proposition \ref{abelian objects in PreOrdGrp} is the following remark, that was first observed in \cite{CMFM}:

\begin{remark} \label{PreOrdGrp is not strongly unital}
The category $\mathsf{PreOrdGrp}$ of preordered groups is not strongly unital.
\end{remark}

\begin{proof}
This follows from Propositions \ref{commutative objects in PreOrdGrp}, \ref{abelian objects in PreOrdGrp}, and \ref{commutative objects in strongly unital categories}. For instance, the inclusion 
$ {\mathbb N} \hookrightarrow  {\mathbb Z}$ of the monoid $\mathbb N$ of non-negative integers in the abelian group $\mathbb Z$ of integers is an example of a  commutative object in $\mathsf{PreOrdGrp}$ that is not an abelian object.
\end{proof}

\begin{corollary}
The category $\mathsf{PreOrdGrp}$ is not \emph{subtractive} (in the sense of \cite{JZ}).
\end{corollary}
\begin{proof}
This follows from the Remark \ref{PreOrdGrp is not strongly unital} and the well-known fact that 
\begin{center}
strongly unital = unital + subtractive
\end{center}
(see \cite{GR}, for instance).
\end{proof}

We will now consider the functor $F$ from the category $\mathsf{PreOrdGrp}$ of preordered groups to its subcategory $\mathsf{Mono}(\mathsf{Ab})$ of abelian objects. We will see it as the composite of the following two functors
\begin{center}
\begin{tikzcd}
\mathsf{PreOrdGrp} \arrow[r,"C"] 
& \mathsf{PreOrdAb} \arrow[r,"A"]
& \mathsf{Mono}(\mathsf{Ab})
\end{tikzcd}
\end{center}
where $C$ is defined as in Section \ref{Functor C} and $A$ is defined, for any preordered abelian group $(G,P_G)$, by
\[A(G,P_G) = (G,grp(P_G))\]
with $grp(P_G)$ the \emph{group completion} of the monoid $P_G$. As a consequence, for any $(G,P_G) \in \mathsf{PreOrdGrp}$,
\[F(G,P_G) = (ab(G),grp(\eta_G(P_G))).\]

\begin{equation}\label{diag group completion}
\begin{tikzcd}[row sep=huge, column sep=huge]
P_G \arrow[r,two heads,"{\bar{\eta}_G}"] \arrow[rr,bend left,"{\hat{\eta}_G}"] \arrow[d,tail]
& \eta_G(P_G) \arrow[r,dotted,"{j_G}"] \arrow[d,tail]
& grp(\eta_G(P_G)) \arrow[d,dotted,tail,"{i_{ab(G)}}"]\\
G \arrow[r,two heads,"{\eta_G}"'] \arrow[rr,bend right,two heads,"{\eta_G}"']
& ab(G) \arrow[r,equal]
& ab(G)
\end{tikzcd}
\end{equation}

Let us recall the general construction of the \emph{group completion} (also called \emph{Grothendieck group}) of an additive commutative monoid $M$. On the cartesian product $M \times M$ one defines an equivalence relation $\sim$ in the following way: $(m_1, m_2) \sim (n_1,n_2)$ if and only if there exists an element $k$ in $M$ such that $m_1 + n_2 + k = m_2 + n_1 + k$. The Grothendieck group $grp(M)$ of $M$ is then given by the quotient $(M \times M)/\sim$, which turns out to be an abelian group. Note that there is a monoid homomorphism $j \colon M \rightarrow grp(M)$ sending any element $m$ of $M$ to the equivalence class $[(m,0)]$ (with respect to $\sim$). This homomorphism satisfies a universal property: for any monoid homomorphism $\phi \colon M \rightarrow X$ from $M$ to an abelian group $X$, there is a unique group homomorphism $\psi \colon grp(M) \rightarrow X$ such that $\phi = \psi \cdot j$.

This universal property yields the existence of the unique morphism $i_{ab(G)}$ making the diagram \eqref{diag group completion} commute. The following lemma implies that $grp(\eta_G(P_G))$ is in addition a submonoid of $ab(G)$, so that $(ab(G),grp(\eta_G(P_G)))$ is then indeed a preordered group.

\begin{lemma}
Let $M$ be a submonoid of an abelian group $X$ and $$Y =  \{x \in X  \mid x = a-b, \, \mathrm{for} \,  a, b \in M\},$$ which is a (normal) subgroup of $X$. Then there is a group isomorphism
\[grp(M) \cong  Y \]
which implies that $grp(M)$ is a (normal) subgroup of $X$.
\end{lemma}  

\begin{proof}
By definition, $grp(M) = (M \times M)/\sim$. Now, we observe that, for $(m_1,m_2), (n_1,n_2) \in M \times M$, $(m_1,m_2) \sim (n_1,n_2)$ if and only if there exists a $k$ in $M$ such that $m_1 + n_2 + k = m_2 + n_1 + k$. We can see this equality in the group $X$ so that, by the cancellation property, $(m_1,m_2) \sim (n_1,n_2)$ if and only if $m_1 + n_2 = m_2 + n_1$, which is equivalent to $m_1 - m_2 = n_1 - n_2$. Accordingly, the group homomorphism $\Phi \colon grp(M) \rightarrow Y$ defined, for any $[(m_1,m_2)]$ in $grp(M)$, by $\Phi([(m_1,m_2)]) = m_1 - m_2$, is an isomorphism.
\end{proof}

%A Galois theory for monoids
\section{The Galois theory corresponding to the group completion functor } \label{Galois theory for monoids}

If we take a look at the restriction of the functor $A$ to the positive cones (i.e. to the category $\mathsf{ComMon}$ of commutative monoids), we then get the \emph{group completion functor} $grp$ studied in \cite{MRVdL}. In this section, we recall the results of this article that will be useful for our work.

In their paper \cite{MRVdL}, Montoli, Rodelo and Van der Linden study the adjunction
\begin{center}
\begin{tikzcd}
\mathsf{Mon} \arrow[rr,bend left=15,"{grp}"]
& \bot
& \mathsf{Grp} \arrow[ll,bend left=15,"{mon}"]
\end{tikzcd}
\end{center}
between the categories $\mathsf{Mon}$ of monoids and $\mathsf{Grp}$ of groups, where the right adjoint $mon$ is the forgetful functor while the left adjoint $grp$ is the \emph{group completion functor}. 

A significant part of the article \cite{MRVdL} is devoted to the proof that the Galois structure
\[\Gamma_{grp} = (\mathsf{Mon}, \mathsf{Grp}, grp, mon, \E_{grp}, \z_{grp})\]
is admissible when $\E_{grp}$ and $\z_{grp}$ are the classes of surjective homomorphisms in $\mathsf{Mon}$ and $\mathsf{Grp}$, respectively. At the end of the paper, a characterization of ($\Gamma_{grp}$-)normal and ($\Gamma_{grp}$-)central extensions is given: both notions coincide with the one of \emph{special homogeneous surjection} that we are now going to recall.

\begin{definition} \label{definition homogeneous split epi} \cite{BMFMS}
Let $f \colon X \rightarrow Y$ be a split epimorphism in $\mathsf{Mon}$, with section $s$ and kernel $k$:
\begin{equation} \label{split epi}
\begin{tikzcd}
K \arrow[r,tail,"k"]
& X \arrow[r,shift left=0.1cm,"f"]
& Y. \arrow[l,shift left=0.1cm,"s"]
\end{tikzcd}
\end{equation}
The split epimorphism $(f,s)$ is \emph{homogeneous} when, for any $y \in Y$, the functions $\mu_y \colon K \rightarrow f^{-1}(y)$ and $\nu_y \colon K \rightarrow f^{-1}(y)$, defined, for any $x \in K$, by $\mu_y(x) = x + s(y)$ and $\nu_y(x) = s(y) + x$, are bijective.
\end{definition}

\begin{definition} \label{definition special homogeneous surjection} \cite{BMFMS}
Let 
\begin{tikzcd}
f \colon X \arrow[r,two heads]
& Y
\end{tikzcd}
be a surjective homomorphism in the category $\mathsf{Mon}$ of monoids. Consider its kernel pair $(Eq(f),\pi_1,\pi_2)$ with the diagonal $\Delta$:
\begin{center}
\begin{tikzcd}
Eq(f) \arrow[r,shift left=0.22cm,"{\pi_1}"] \arrow[r,shift right=0.22cm,"{\pi_2}"']
& X \arrow[l,"\Delta" description] \arrow[r,two heads,"f"]
& Y.
\end{tikzcd}
\end{center}
The morphism $f$ is said to be a \emph{special homogeneous surjection} when $(\pi_1,\Delta)$ is a homogeneous split epimorphism.
\end{definition}

Special homogeneous surjections are pullback-stable and, moreover, one has the following property:

\begin{proposition} \label{SHS reflect} \cite{BMFMS}
Consider in $\mathsf{Mon}$ the following pullback where both $f$ and $f'$ are surjective homomorphisms:
\begin{center}
\begin{tikzcd}
P \arrow[r,two heads,"f'"] \arrow[d,two heads,"g'"']
& Z \arrow[d,two heads,"g"] \\
X \arrow[r,two heads,"f"']
& Y.
\end{tikzcd}
\end{center}
If $g'$ is a special homogeneous surjection, then so is $g$.
\end{proposition}

We now recall in Theorem \ref{Gamm_grp normal and central extensions} the characterization of ($\Gamma_{grp}$-)normal and ($\Gamma_{grp}$-)central extensions considered in \cite{MRVdL}.
 
\begin{proposition} \label{Gamma_grp trivial extensions that are split epis} \cite{MRVdL}
Consider an arbitrary split epimorphism $(f,s)$ as in \eqref{split epi}. Then the following conditions are equivalent:
\begin{enumerate}
\item $f$ is a $(\Gamma_{grp}\text{-})$trivial extension.
\item $f$ is a special homogeneous surjection.
\end{enumerate}
\end{proposition}

\begin{theorem} \label{Gamm_grp normal and central extensions} \cite{MRVdL}
Let 
\begin{tikzcd}
f \colon X \arrow[r,two heads]
& Y
\end{tikzcd}
be a surjective homomorphism of monoids. Then the following conditions are equivalent:
\begin{enumerate}
\item $f$ is a special homogeneous surjection.
\item $f$ is a $(\Gamma_{grp}\text{-})$normal extension.
\item $f$ is a $(\Gamma_{grp}\text{-})$central extension.
\end{enumerate}
\end{theorem}

\section{The functor $F \colon \mathsf{PreOrdGrp} \rightarrow \mathsf{Mono(Ab)}$ and its induced admissible Galois structure}

\begin{proposition}
There is an adjunction
\begin{equation} \label{adjunction PreOrdGrp_Ab(PreOrdGrp)}
\begin{tikzcd}
\mathsf{PreOrdGrp} \arrow[rr,bend left=7.5,"F"]
& \bot
&\mathsf{Mono(Ab)} \arrow[ll,hook',bend left=7.5,"U"]
\end{tikzcd}
\end{equation}
between the category $\mathsf{PreOrdGrp}$ of preordered groups and its full subcategory $\mathsf{Mono(Ab)}$ of abelian objects, where the right adjoint $U$ is the inclusion functor and the left adjoint $F$ is defined as in Section \ref{Commutative and abelian objects}: for any $(G,P_G) \in \mathsf{PreOrdGrp}$,
\[F(G,P_G) = (ab(G) = G/[G,G],grp(\eta_G(P_G))).\]
%For any $(G,P_G) \in \mathsf{PreOrdGrp}$, the $(G,P_G)$-component of the unit of this adjunction is given by the morphism $(\eta_G,\hat{\eta}_G)$ in the following commutative diagram:
%\begin{center}
%\begin{tikzcd}[row sep=huge, column sep=huge]
%P_G \arrow[r,two heads,"{\bar{\eta}_G}"] \arrow[rr,bend left,"{\hat{\eta}_G}"] \arrow[d,tail]
%& \eta_G(P_G) \arrow[r,tail,"{j_G}"] \arrow[d,tail]
%& grp(\eta_G(P_G)) \arrow[d,tail]\\
%G \arrow[r,two heads,"{\eta_G}"'] \arrow[rr,bend right,two heads,"{\eta_G}"']
%& ab(G) \arrow[r,equal]
%& ab(G).
%\end{tikzcd}
%\end{center}
\end{proposition}

\begin{proof}
The adjunction \eqref{adjunction PreOrdGrp_Ab(PreOrdGrp)} can be seen as the composite of the two adjunctions 
\begin{center}
\begin{tikzcd}
\mathsf{PreOrdGrp} \arrow[rr,bend left=7.5,"C"]
& \bot
& \mathsf{PreOrdAb} \arrow[ll,hook',bend left=7.5,"V"] \arrow[rr,bend left=7.5,"A"]
& \bot
& \mathsf{Mono(Ab)} \arrow[ll,hook',bend left=7.5,"W"]
\end{tikzcd}
\end{center}
where the left-hand one has been studied in Section \ref{Functor C} and (the restriction to the positive cones of) the right-hand one has been considered in Section \ref{Galois theory for monoids}. Note that the $(G,P_G)$-component of the unit of the composite adjunction is given by the morphism $(\eta_G,\hat{\eta}_G)$ (in the notations of diagram \eqref{diag group completion}).
\end{proof}

\begin{remark}
\emph{For a given preordered group $(G,P_G)$, the $(G,P_G)$-component $(\eta_G,\hat{\eta}_G)$ of the unit of the adjunction \eqref{adjunction PreOrdGrp_Ab(PreOrdGrp)} is not a regular epimorphism, in general. The morphism $(\eta_G,\hat{\eta}_G)$ is just an epimorphism, since $\eta_G$ is surjective, while $\hat{\eta}_G$ is not necessarily surjective. Note also that the morphism $j_G \colon  \eta_G(P_G) \rightarrow grp( \eta_G(P_G))$ is a monomorphism because ${\eta}_G(P_G)$ is a commutative monoid with cancellation.}
\end{remark}

\begin{proposition}
If $\E$ denotes the class of regular epimorphisms in $\mathsf{PreOrdGrp}$ and $\z$ the class of regular epimorphisms in $\mathsf{Mono(Ab)}$, then
\begin{equation} \label{Galois structure Gamma}
\Gamma = (\mathsf{PreOrdGrp}, \mathsf{Mono(Ab)}, F, U, \E, \z)
\end{equation}
is a Galois structure.
\end{proposition}

Since the Galois structure \eqref{Galois structure Gamma} is a composite of two “compatible” Galois structures that are admissible, we may use the following known result in order to prove its admissibility (the argument given to prove Lemma $6.2$ in \cite{DEG}, for instance, still holds in our situation).

\begin{proposition} \label{composite of admissible Galois structures}
Consider the following chain of adjunctions
\begin{equation}\label{chain of adjunctions}
\begin{tikzcd}
\C \arrow[rr,bend left=15,"J"]
& \bot
& \B \arrow[ll,hook',bend left=15,"L"] \arrow[rr,bend left=15,"I"]
& \bot
& \A \arrow[ll,hook',bend left=15,"H"]
\end{tikzcd}
\end{equation}
where $\A$ is a full subcategory of $\B$ and $\B$ a full subcategory of $\C$. Assume moreover that
\begin{itemize}
\item $\Gamma_J = (\C,\B,J,L,\E_J,\z_J)$;
\item $\Gamma_I = (\B,\A,I,H,\E_I,\z_I)$
\end{itemize}
are admissible Galois structures that are “compatible” in the sense that $J(\E_J) {\subset} \E_I$ and $H(\z_I) {\subset} \z_J$. Then the composite of these two admissible Galois structures
\[\Gamma = (\C,\A,I \cdot J,L \cdot H, \E, \z),\]
where $\E = \E_J$ and $\z = \z_I$, is an admissible Galois structure.
\end{proposition}

\begin{corollary}
The Galois structure \eqref{Galois structure Gamma} is admissible.
\end{corollary}

%Gamma-normal and Gamma-central extensions
\section{Characterization of $\Gamma$-normal and $\Gamma$-central extensions}

This section is devoted to $\Gamma$-normal and $\Gamma$-central extensions which will be characterized in Theorem \ref{Gamma-normal and Gamma-central extensions}. The following two lemmas will be needed for the proof of this result.

\begin{lemma} \label{trivial extensions for a composite}
Consider a chain of adjunctions as in \eqref{chain of adjunctions} and assume that we have admissible Galois structures $\Gamma_J$ and $\Gamma_I$ as in Proposition \ref{composite of admissible Galois structures}. Assume, moreover, that any component of the unit of the adjunction $J \dashv L$ is a descent morphism. Then, for an extension  $f \colon X \rightarrow Y$ in $\C$, the following conditions are equivalent:
\begin{enumerate}
\item $f$ is a $\Gamma$-trivial extension for the admissible Galois structure $\Gamma$ described in Proposition \ref{composite of admissible Galois structures}.
\item $f$ is a $\Gamma_J$-trivial extension and $J(f)$ is a $\Gamma_I$-trivial extension.
\end{enumerate}
\end{lemma}

\begin{proof}
\begin{itemize}
\item $(1) \Rightarrow (2)$: By assumption, the square
\begin{center}
\begin{tikzcd}
X \arrow[r,"{\eta_X}"] \arrow[d,"f"']
& (I \cdot J)(X) \arrow[d,"{(I \cdot J)(f)}"]\\
Y \arrow[r,"{\eta_Y}"']
& (I \cdot J)(Y)
\end{tikzcd}
\end{center}
is a pullback in $\C$. Note that this pullback decomposes into the following two squares:
\begin{equation} \label{diagram 2 squares}
\begin{tikzcd}[row sep=large,column sep=large]
X \arrow[r,"{\eta_{X}^{J}}"] \arrow[d,"f"']
& J(X) \arrow[r,"{\eta_{J(X)}^{I}}"] \arrow[d,"{J(f)}"]
& (I \cdot J)(X) \arrow[d,"{(I \cdot J)(f)}"]\\
Y \arrow[r,"{\eta_{Y}^{J}}"']
& J(Y) \arrow[r,"{\eta_{J(Y)}^{I}}"']
& (I \cdot J)(Y)
\end{tikzcd}
\end{equation}
(where $\eta^J$ and $\eta^{I}$ are the units of the adjunctions $J \dashv L$ and $I \dashv H$, respectively). Equivalently, the fact that $f$ is a $\Gamma$-trivial extension can also be formulated by saying that 
\[f = (L \cdot H)^Y(g)\]
for some extension $g \colon A \rightarrow (I \cdot J)(Y)$ in $\A$. Now,
\[(L \cdot H)^Y(g) = L^Y(H^{J(Y)}(g)),\]
so that $f = L^Y(\bar{g})$, for an extension $\bar{g} = H^{J(Y)}(g)$ in $\B$ with codomain $J(Y)$. This precisely means that $f$ is a $\Gamma_J$-trivial extension, hence in the diagram \eqref{diagram 2 squares} both the left-hand square and the external rectangle are pullbacks. Since $\eta_{Y}^{J}$ is a descent morphism, it then follows that also the right-hand square is a pullback, meaning that $J(f)$ is a $\Gamma_I$-trivial extension.
\item $(2) \Rightarrow (1)$: Statement $(2)$ means that both squares in diagram \eqref{diagram 2 squares} are pullbacks, so that the external rectangle is also a pullback, and $f$ is a $\Gamma$-trivial extension. \qedhere
\end{itemize}
\end{proof}

\begin{lemma} \label{special homogeneous surjections in PreOrdGrp}
Let 
\begin{tikzcd}
(f,\bar{f}) \colon (G,P_G) \arrow[r,two heads]
& (H,P_H)
\end{tikzcd}
be a regular epimorphism in $\mathsf{PreOrdGrp}$. Then $\bar{f}$ is a special homogeneous surjection in $\mathsf{Mon}$ if and only if, for any $(x,y) \in Eq(\bar{f})$, $y - x \in P_G$ and $-x + y \in P_G$.
\end{lemma}

\begin{proof}
The surjective homomorphism $\bar{f}$ is special homogeneous precisely when 
\begin{tikzcd}
Eq(\bar{f}) \arrow[r,shift left,"{\bar{\pi}_1}"]
& P_G \arrow[l,shift left,"{\bar{\Delta}}"] 
\end{tikzcd}
is a homogeneous split epimorphism (see Definition \ref{definition special homogeneous surjection}). This happens if and only if,
for any $x \in P_G$, the functions $$\mu_x \colon \mathsf{Ker}(\bar{\pi}_1) \cong \mathsf{Ker}(\bar{f}) \rightarrow \bar{\pi}_{1}^{-1}(x); (0,a) \mapsto (0,a) + \bar{\Delta}(x) = (x,a+x)$$ and $$\nu_x \colon \mathsf{Ker}(\bar{\pi}_1) \cong \mathsf{Ker}(\bar{f}) \rightarrow \bar{\pi}_{1}^{-1}(x); (0,b) \mapsto \bar{\Delta}(x) + (0,b) = (x,x + b)$$ are bijective. This is equivalent to the fact that, for any $(x,y) \in Eq(\bar{f})$, there exist a unique $a \in \mathsf{Ker}(\bar{f})$ and a unique $b \in \mathsf{Ker}(\bar{f})$ such that $(x,y) = (x,a+x)$ and $(x,y) = (x,x + b)$.
This condition also amounts to asking the existence of a unique $a \in \mathsf{Ker}(f) \cap P_G$ and a unique $b \in \mathsf{Ker}(f) \cap P_G$ such that $y = a + x$ and $y = x + b$, so that $a$ and $b$ are elements of $G$ of the form $a = y - x$ and $b = -x + y$. 
As a consequence, the surjective homomorphism $\bar{f}$ is special homogeneous if and only if, for any $(x,y) \in Eq(\bar{f})$, $y - x \in P_G$ and $-x + y \in P_G$.
\end{proof}

\begin{theorem} \label{Gamma-normal and Gamma-central extensions}
Let
\begin{tikzcd}
(f,\bar{f}) \colon (G,P_G) \arrow[r,two heads]
& (H,P_H)
\end{tikzcd}
be a regular epimorphism in $\mathsf{PreOrdGrp}$. Then the following conditions are equivalent:
\begin{enumerate}
\item 
\begin{itemize}
\item[(a)] $\mathsf{Ker}(f) \subseteq Z(G)$;
\item[(b)] $\bar{f}$ is a special homogeneous surjection in $\mathsf{Mon}$.
\end{itemize}
\item $(f,\bar{f})$ is a $(\Gamma\text{-})$normal extension.
\item $(f,\bar{f})$ is a $(\Gamma\text{-})$central extension.
\end{enumerate}
\end{theorem}

\begin{proof}
\begin{itemize}
\item $(1) \Rightarrow (2)$: First remark that condition (b) implies that $a - b + c \in P_G$ for any $(a,b,c) \in Eq(\bar{\eta}_G) \times_{P_G} Eq(\bar{f})$. Indeed, thanks to Lemma \ref{special homogeneous surjections in PreOrdGrp}, we know that $-b + c \in P_G$ since $(b,c) \in Eq(\bar{f})$, and then $a - b + c \in P_G$ since $a \in P_G$, and $P_G$ is a submonoid of $G$. So condition (b) implies condition (ii) of Theorem \ref{GammaC-normal extensions} and it then follows, thanks to the validity of condition (a), that $(f,\bar{f})$ is a $\Gamma_C$-normal extension. In other words, the first projection \begin{tikzcd}
(\pi_1, \bar{\pi}_1) \colon (Eq(f),Eq(\bar{f})) \arrow[r,two heads]
& (G,P_G),
\end{tikzcd}
of the kernel pair of $(f,\bar{f})$ is a $\Gamma_C$-trivial extension. According to Lemma \ref{trivial extensions for a composite}, it now remains to prove that $C(\pi_1,\bar{\pi}_1)$ is a $\Gamma_A$-trivial extension (where $\Gamma_A$ denotes the admissible Galois structure associated with the reflection $A \dashv W$). By condition (b), $\bar{f}$ is a special homogeneous surjection, which implies that also $\bar{\pi}_1$ is a special homogeneous surjection by pullback-stability. This in turn implies that the restriction $C(\bar{\pi}_1)$ of $C(\pi_1,\bar{\pi}_1)$ to the positive cones is a special homogeneous surjection thanks to Proposition \ref{SHS reflect}. Indeed, the square
\begin{center}
\begin{tikzcd}[row sep = large, column sep = large]
Eq(\bar{f}) \arrow[r,two heads,"{\bar{\eta}_{Eq(f)}}"] \arrow[d,two heads,"{\bar{\pi}_1}"']
& \eta_{Eq(f)}(Eq(\bar{f})) \arrow[d,two heads,"{C(\bar{\pi}_1)}"]\\
P_G \arrow[r,two heads,"{\bar{\eta}_G}"']
& \eta_G(P_G)
\end{tikzcd}
\end{center}
is a pullback of regular epimorphisms  in $\mathsf{Mon}$, since $(\pi_1,\bar{\pi}_1)$ is a $\Gamma_C$-trivial extension. Accordingly, $C(\bar{\pi}_1)$ is a split epimorphism which is a special homogeneous surjection, and is then a $\Gamma_{grp}$-trivial extension
by Proposition \ref{Gamma_grp trivial extensions that are split epis}. This means that the square
\begin{center}
\begin{tikzcd}[row sep = large,column sep=large]
\eta_{Eq(f)}(Eq(\bar{f})) \arrow[r,tail,"{j_{Eq(f)}}"] \arrow[d,two heads,"{C(\bar{\pi}_1)}"']
& grp(\eta_{Eq(f)}(Eq(\bar{f}))) \arrow[d,two heads,"{grp(C(\bar{\pi}_1))}"]\\
\eta_G(P_G) \arrow[r,tail,"{j_G}"']
& grp(\eta_G(P_G))
\end{tikzcd}
\end{center}
is a pullback in the category $\mathsf{ComMon}$ of commutative monoids. As a consequence, the square
\begin{center}
\begin{tikzcd}[row sep = huge,column sep=huge]
(ab(Eq(f)),\eta_{Eq(f)}(Eq(\bar{f}))) \arrow[r,tail,"{(1,j_{Eq(f)})}"] \arrow[d,two heads,"{C(\pi_1,\bar{\pi}_1)}"']
& (ab(Eq(f)),grp(\eta_{Eq(f)}(Eq(\bar{f})))) \arrow[d,two heads,"{(A \cdot C)(\pi_1,\bar{\pi}_1) = F(\pi_1,\bar{\pi}_1)}"]\\
(ab(G),\eta_G(P_G)) \arrow[r,tail,"{(1,j_G)}"']
& (ab(G),grp(\eta_G(P_G)))
\end{tikzcd}
\end{center}
is a pullback in the category $\mathsf{PreOrdAb}$ of preordered abelian groups, and this proves that $C(\pi_1,\bar{\pi}_1)$ is a $\Gamma_A$-trivial extension. By Lemma \ref{trivial extensions for a composite}, we conclude that $(\pi_1,\bar{\pi}_1)$ is a $\Gamma$-trivial extension, which is equivalent to saying that $(f,\bar{f})$ is a $\Gamma$-normal extension.
\item $(2) \Rightarrow (3)$: Any normal extension is central by definition.
\item $(3) \Rightarrow (1)$: By definition of a ($\Gamma$-)central extension, there exists an effective descent morphism (i.e. a regular epimorphism)
\begin{tikzcd}
(p,\bar{p}) \colon (E,P_E) \arrow[r,two heads]
& (H,P_H)
\end{tikzcd} 
in $\mathsf{PreOrdGrp}$ such that the pullback $(p,\bar{p})^*(f,\bar{f})$ (which we will denote by $(\pi_1,\bar{\pi}_1)$) of $(f,\bar{f})$ along $(p,\bar{p})$ is a $\Gamma$-trivial extension. Using Lemma \ref{trivial extensions for a composite}, this means that
\begin{itemize}
\item[$(\alpha)$] $(\pi_1,\bar{\pi}_1)$ is a $\Gamma_C$-trivial extension, and
\item[$(\beta)$] $C(\pi_1,\bar{\pi}_1)$ is a $\Gamma_A$-trivial extension. 
\end{itemize}
The first statement $(\alpha)$ means of course that $(f,\bar{f})$ is a $\Gamma_C$-central extension. In particular, $f$ is algebraically central, that is,
\begin{tikzcd}
f \colon G \arrow[r,two heads]
& H
\end{tikzcd}
satisfies condition (a). The second statement $(\beta)$ now implies that the square
\begin{center}
\begin{tikzcd}[row sep = huge,column sep=huge]
(ab(P),\eta_{P}(P_P)) \arrow[r,tail,"{(1,j_{P})}"] \arrow[d,two heads,"{(C(\pi_1),C(\bar{\pi}_1)) = C(\pi_1,\bar{\pi}_1)}"']
& (ab(P),grp(\eta_{P}(P_P))) \arrow[d,two heads,"{(A \cdot C)(\pi_1,\bar{\pi}_1) = F(\pi_1,\bar{\pi}_1)}"]\\
(ab(E),\eta_E(P_E)) \arrow[r,tail,"{(1,j_E)}"']
& (ab(E),grp(\eta_E(P_E)))
\end{tikzcd}
\end{center}
(where $P = E \times_H G$ and $P_P = P_E \times_{P_H} P_G$) is a pullback in $\mathsf{PreOrdAb}$. In particular, its restriction to $\mathsf{ComMon}$
\begin{center}
\begin{tikzcd}[row sep = large,column sep=large]
\eta_{P}(P_P) \arrow[r,tail,"{j_P}"] \arrow[d,two heads,"{C(\bar{\pi}_1)}"']
& grp(\eta_{P}(P_P)) \arrow[d,two heads,"{grp(C(\bar{\pi}_1))}"]\\
\eta_E(P_E) \arrow[r,tail,"{j_E}"']
& grp(\eta_E(P_E))
\end{tikzcd}
\end{center}
is a pullback. This means that $C(\bar{\pi}_1)$ is a $\Gamma_{grp}$-trivial extension, and then a $\Gamma_{grp}$-normal extension. Indeed, since $\Gamma_{grp}$ is admissible, any $\Gamma_{grp}$-trivial extension is $\Gamma_{grp}$-normal. Thanks to Theorem \ref{Gamm_grp normal and central extensions}, we can conclude that $C(\bar{\pi}_1)$ is a special homogeneous surjection. Since special homogeneous surjections are pullback-stable, it then follows that $\bar{\pi}_1$ has this same property. Indeed, the statement $(\alpha)$ implies (among other things) that the square
\begin{center}
\begin{tikzcd}
P_P \arrow[r,two heads,"{\bar{\eta}_P}"] \arrow[d,two heads,"{\bar{\pi}_1}"']
& \eta_P(P_P) \arrow[d,two heads,"{C(\bar{\pi}_1)}"]\\
P_E \arrow[r,two heads,"{\bar{\eta}_E}"']
& \eta_E(P_E)
\end{tikzcd}
\end{center}
is a pullback in $\mathsf{Mon}$. Applying now Proposition \ref{SHS reflect} to the following pullback of regular epimorphisms
\begin{center}
\begin{tikzcd}
P_P \arrow[r,two heads,"{\bar{\pi}_2}"] \arrow[d,two heads,"{\bar{\pi}_1}"']
& P_G \arrow[d,two heads,"{\bar{f}}"]\\
P_E \arrow[r,two heads,"{\bar{p}}"']
& P_H,
\end{tikzcd}
\end{center}
we obtain that $\bar{f}$ is a special homogeneous surjection, i.e. condition (b) is satisfied. This concludes the proof. \qedhere
\end{itemize}
\end{proof}

\end{document}